\documentclass[reqno]{amsart} 
\usepackage{hyperref} 
\usepackage{amssymb,amsmath,amsthm}
\usepackage{amsfonts}
\usepackage{graphicx,epsfig}

\newtheorem{theorem}{Theorem}
\newtheorem{conj}{Conjecture}
\newtheorem{lemma}{Lemma} 

\newtheorem{claim}[lemma]{Claim}
\newtheorem{prob}{Problem}
\theoremstyle{definition}

\def\Z{{\mathbb{Z}}} 
\def\N{{\mathbb{N}}} 
\def\R{{\mathbb{R}}} 
\def\cA{{\mathcal A}}
\def\cB{{\mathcal B}}
\def\cF{{\mathcal F}}
\def\cG{{\mathcal G}}

\def\aa{\bar\alpha}
\def\bb{\bar\beta}
\def\ee{{\epsilon}}
\def\ts{\textstyle}

\title{When are stars the largest cross-intersecting families?}

\author{Norihide Tokushige}
\address{College of Education, Ryukyu University, Nishihara, Okinawa 903-0213, Japan}
\email{hide@edu.u-ryukyu.ac.jp}
\thanks{The author was supported by JSPS KAKENHI 25287031}


\begin{document}

\begin{abstract}
We provide a necessary and sufficient condition for stars to be the
largest cross-intersecting families.
\end{abstract} 

\maketitle

\section{Introduction}
Let $2^{[n]}$ denote the power set of $[n]:=\{1,2,\ldots, n\}$, and 
let $\binom{[n]}k$ denote the set of all $k$-element subsets of $[n]$.
A family of subsets $\cA\subset 2^{[n]}$ is \emph{intersecting} if
$A\cap A'\neq\emptyset$ for all $A,A'\in\cA$.
If $n<2k$ then $\binom{[n]}k$ itself is intersecting. But if $n\geq 2k$
then $\binom{[n]}k$ is no longer intersecting, and we can ask for the maximum
size of intersecting families in $\binom{[n]}k$.
The Erd\H os--Ko--Rado Theorem answers this question, which tells us that
if $n\geq 2k$ and $\cA\subset\binom{[n]}k$ is intersecting,
then $|\cA|\leq\binom{n-1}{k-1}$. 
Equality comes from a star, that is,
$\cA=\{A\in\binom{[n]}k:i\in A\}$ for some fixed element $i\in[n]$,
and moreover this is the only optimal configuration unless $n=2k$.
In this paper we attempt to extend the result to two families.

Two families $\cA$ and $\cB$ in $2^{[n]}$ are \emph{cross-intersecting}
if $A\cap B\neq\emptyset$ for all $A\in\cA$ and $B\in\cB$.
It is then natural to ask the following.
\begin{prob}
Let $n,k,l$ be positive integers. 
What is the maximum of the product $|\cA||\cB|$ among
cross-intersecting families $\cA\subset\binom{[n]}k$ and 
$\cB\subset\binom{[n]}l$?
\end{prob}
Let $M(n,k,l)$ denote the maximum.
If $k+l>n$ then $M(n,k,l)$ is clearly $\binom nk\binom nl$. 
It is not difficult to see that if $k+l=n$ then 
$M(n,k,l)=\lfloor\frac12\binom nk\rfloor\lceil\frac12\binom nl\rceil$.\footnote{
Note that $A\in\cA$ implies
$[n]\setminus A\not\in\cB$, and $|\cB|\leq\binom nl-|\cA|=\binom nk-|\cA|$.
}
The first non-trivial result was obtained by Pyber \cite{P}, and later extended
by Matsumoto and the author \cite{MT}, 
which states that if $n\geq\max\{2k,2l\}$ then
\begin{align}\label{EKRbound}
M(n,k,l)=\binom{n-1}{k-1}\binom{n-1}{l-1}. 
\end{align}
Moreover the stars are the only optimal configurations if $n>\max\{2k,2l\}$.
Bey \cite{Bey} gave an alternative combinatorial proof, and 
Suda and Tanaka \cite{ST} established a semidefinite programming approach 
to obtain related results including \eqref{EKRbound}.
These are perhaps all the results known about $M(n,k,l)$.
It is open for the remaining cases $k+l<n$ and $n<\max\{2k,2l\}$. These cases
look more interesting and more difficult than the known cases,
for most likely infinitely many different structures appear as optimal families.
This paper addresses the problem of finding the necessary and 
sufficient conditions for $n,k,l$ such that \eqref{EKRbound} holds.
To begin with let us introduce the following cross-intersecting families
$\cA^{(k)}_j$ and $\cB^{(l)}_j$ for $j\geq 0$ defined by
\begin{align*}
\cA_j^{(k)}&:=\{A\in\tbinom{[n]}k:1\in A\}\sqcup
\{A\in\tbinom{[n]}k:A\cap[j+2]=[j+2]\setminus\{1\}\},\\
\cB_j^{(l)}&:=\{B\in\tbinom{[n]}l:1\in B\}\setminus
\{B\in\tbinom{[n]}l:B\cap[j+2]=\{1\}\}.
\end{align*}
To ensure \eqref{EKRbound} it is necessary that
$|\cA^{(k)}_j||\cB^{(l)}_j|$ does not exceed the RHS of \eqref{EKRbound}
for every $j\geq 0$. Indeed we conjecture that this is sufficient as well.
\begin{conj}\label{conj1}
Let $n,k,l$ be positive integers with $k+l<n<\max\{2k,2l\}$.
Suppose that 
$|\cA^{(k)}_j| |\cB^{(l)}_j| < \binom{n-1}{k-1} \binom{n-1}{l-1}$
for all $j\geq 0$. Then $M(n,k,l)=\binom{n-1}{k-1}\binom{n-1}{l-1}$.
Moreover if $\cA\subset\binom{[n]}k$ and $\cB\subset\binom{[n]}l$ are
cross-intersecting families satisfying $|\cA||\cB|=M(n,k,l)$ then
$\cA=\{A\in\binom{[n]}k:i\in A\}$ and 
$\cB=\{B\in\binom{[n]}l:i\in B\}$ for some $i\in[n]$.
\end{conj}

We prove the conjecture under some additional conditions. But before
stating our results let us introduce a measure counterpart of the problem,
which is closely related to the original problem and easier to 
understand the corresponding conditions.
For a real number $p\in(0,1)$ and a family of subsets $\cF\subset 2^{{[n]}}$
we define the measure of the family $\mu_p(\cF)$ by
\[
 \mu_p(\cF):=\sum_{F\in\cF}p^{|F|}(1-p)^{n-|F|}.
\]
Then we can ask the following.
\begin{prob}
Let $n\in\N$ and $\alpha,\beta\in(0,1)$.
What is the maximum of the product $\mu_\alpha(\cA)\mu_\beta(\cB)$ among
cross-intersecting families $\cA, \cB\subset 2^{[n]}$?
\end{prob}
Let $m(n,\alpha,\beta)$ denote the maximum. Since $m(n,\alpha,\beta)\leq 1$
and $m(n,\alpha,\beta)$ is increasing\footnote{
Indeed if $m(n,\alpha,\beta)=\mu_{\alpha}(\cA)\mu_\beta(\cB)$ 
for some $\cA,\cB\subset 2^{[n]}$, then 
$\cA':=\cA\cup\{A\cup\{n+1\}:A\in\cA\}$ and
$\cB':=\cB\cup\{B\cup\{n+1\}:B\in\cB\}$ satisfy
$\mu_{\alpha}(\cA)\mu_\beta(\cB)=\mu_{\alpha}(\cA')\mu_\beta(\cB')$,
implying $m(n+1,\alpha,\beta)\geq m(n,\alpha,\beta)$.
} in $n$, we can also define
\[
 m(\alpha,\beta):=\lim_{n\to\infty}m(n,\alpha,\beta).
\]
Note that $m(n_0,\alpha,\beta)=c$ and $m(\alpha,\beta)=c$ imply
$m(n,\alpha,\beta)=c$ for all $n\geq n_0$.
A simple computation shows that
$m(\alpha,\beta)=1$ if $\alpha+\beta>1$, and
$m(\alpha,\beta)=\frac14$ if $\alpha+\beta=1$.
In \cite{T0} it is shown that if $\max\{\alpha,\beta\}\leq\frac12$ then
\begin{align}\label{pEKRbound}
m(n,\alpha,\beta)=\alpha\beta 
\end{align}
for all $n\geq 1$. Moreover it is shown in \cite{STT} that 
if $\max\{\alpha,\beta\}<\frac12$ and $n\geq 4$, then the only
optimal families are the stars, that is, $\cA=\cB=\{F\in 2^{[n]}:i\in F\}$ 
for some $i$. 
\begin{figure}[h]
\includegraphics[width=6cm]{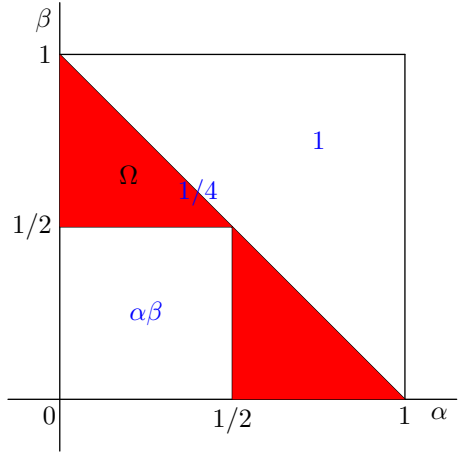}  
\caption{$m(\alpha,\beta)$}\label{fig0}
\end{figure} 
We illustrate the known values of $m(\alpha,\beta)$ in Figure~\ref{fig0}.
The two red triangles indicate the regions where $m(\alpha,\beta)$ is unknown.
By symmetry we focus on the upper left triangle
\[
\Omega:=\{(\alpha,\beta)\in\R^2:
\alpha>0,\, \beta>\frac12, \, \alpha+\beta<1\}. 
\]
In this paper, we provide the necessary and sufficient conditions for 
$(\alpha,\beta)\in\Omega$ satisfying \eqref{pEKRbound}.
A trivial necessary condition comes from the following cross-intersecting
families $\cA_j$ and $\cB_j$ for $j\geq 0$:
\begin{align*}
\cA_j&:=\{A\in 2^{[n]}:1\in A\}\sqcup\{A\in 2^{[n]}:A\cap[j+2]=[j+2]\setminus\{1\}\}\\
\cB_j&:=\{B\in 2^{[n]}:1\in B\}\setminus\{B\in 2^{[n]}:B\cap[j+2]=\{1\}\}.
\end{align*}
If \eqref{pEKRbound} holds, then $\mu_{\alpha}(\cA_j)\mu_\beta(\cB_j)$ 
does not exceed $\alpha\beta$. Our main result states that the converse
is also true.

\begin{theorem}\label{thm-p}
Let $(\alpha,\beta)\in\Omega$. Suppose that 
$\mu_{\alpha}(\cA_j)\mu_\beta(\cB_j)<\alpha\beta$ for all $j\geq 0$.
Then $m(n,\alpha,\beta)=\alpha\beta$ for all $n\geq 1$.
\end{theorem}

\begin{conj}
Let $\alpha$ and $\beta$ satisfy the premises in Theorem~\ref{thm-p}.
If $\cF$ and $\cG$ are cross-intersecting families in $2^{[n]}$
satisfying $\mu_\alpha(\cF)\mu_\beta(\cG)=\alpha\beta$, then
$\cF=\cG=\{F\in 2^{[n]}:i\in F\}$ for some $i\in[n]$.
\end{conj}

To visualize the conditions for $\alpha$ and $\beta$ in Theorem~\ref{thm-p}
let us compute the measures of $\cA_j,\cB_j$.
We have
\begin{align*}
\mu_\alpha(\cA_j)&=\alpha+(1-\alpha)\alpha^{j+1},\\
\mu_\beta(\cB_j)&=\beta-\beta(1-\beta)^{j+1}.
\end{align*}
Thus $\mu_\alpha(\cA_j)\mu_\beta(\cB_j)<\alpha\beta$ if and only if
\begin{align}\label{boundaryCi}
\left(1+(1-\alpha)\alpha^j\right) \left(1-(1-\beta)^{j+1}\right)<1.  
\end{align}
By solving \eqref{boundaryCi} for $\beta$ in terms of $\alpha$ we get
$\beta<e_j$, where 
\begin{align}\label{eq:ej}
e_j=e_j(\alpha):=
1-\left(\frac{\alpha^j-\alpha^{j+1}}{1+\alpha^j-\alpha^{j+1}}\right)^{\frac1{j+1}},
\end{align}
see Figure~\ref{fig1}.
\begin{figure}[h]
\includegraphics[width=12cm]{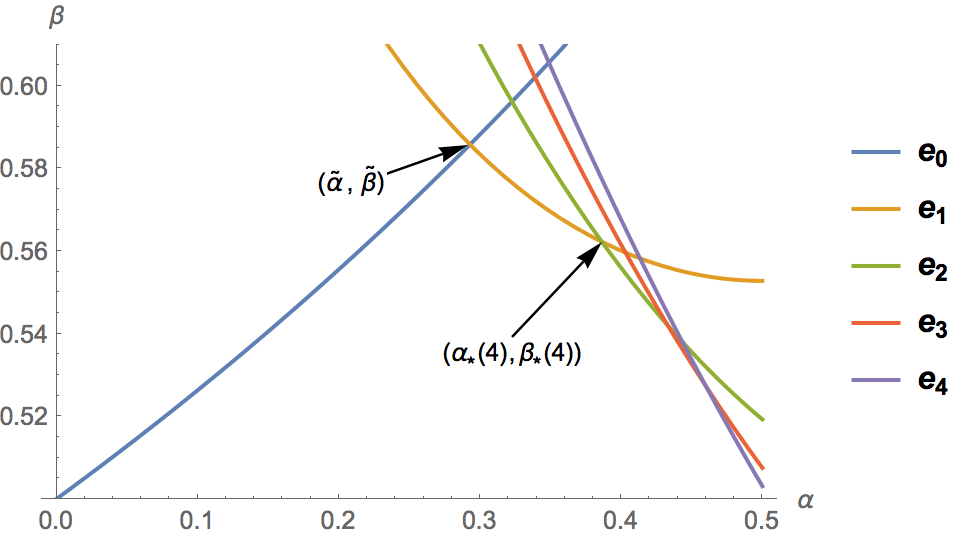}  
\caption{The functions $e_i$}\label{fig1}
\end{figure} 
Then the condition $\mu_\alpha(\cA_j)\mu_\beta(\cB_j)<\alpha\beta$ for all 
$j\geq 0$ is equivalent to 
\begin{align}\label{b<ei}
 \beta<\min\{e_j:j\geq 0\}. 
\end{align}

Let $\Delta\subset\Omega$ be the set of $(\alpha,\beta)$ satisfying \eqref{b<ei},
see Figure~\ref{fig2}. Theorem~\ref{thm-p} tells us that if 
$(\alpha,\beta)\in \Delta$ then $m(\alpha,\beta)=\alpha\beta$, and
conversely if $(\alpha,\beta)\in\Omega$ and 
$m(\alpha,\beta)=\alpha\beta$, then $(\alpha,\beta)$ is in $\Delta$ or its
boundary.

\begin{figure}[h]
\includegraphics[width=11cm]{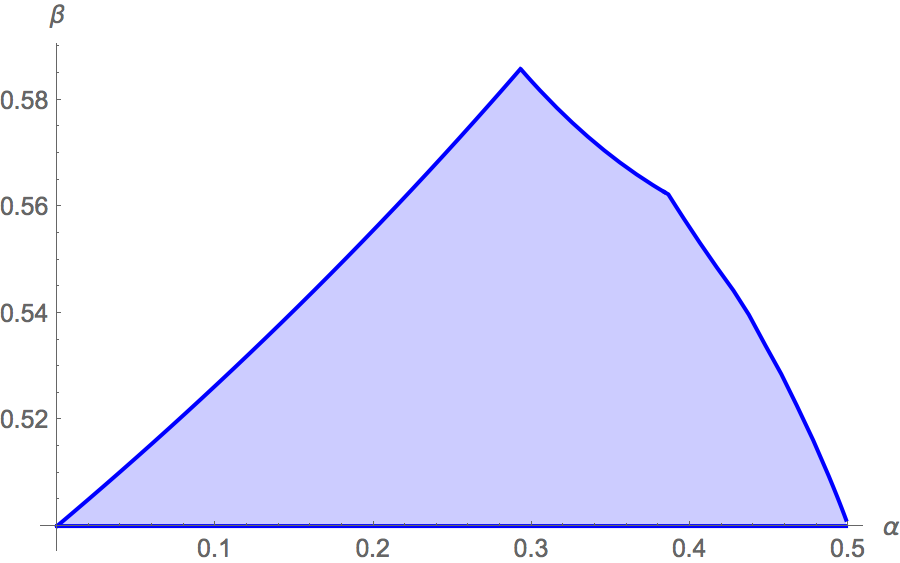}  
\caption{The domain $\Delta$}\label{fig2}
\end{figure} 

We will derive Theorem~\ref{thm-p} from the following result which is 
a partial solution to Conjecture~\ref{conj1}.

\begin{theorem}\label{thm-unif}
Let $(\alpha,\beta)\in\Delta$ be fixed.
Then there exists $n_0=n_0(\alpha,\beta)$
such that if $n>n_0$ then $M(n,k,l)=\binom{n-1}{k-1}\binom{n-1}{l-1}$,
where $k=\lfloor \alpha n\rfloor$ and $l=\lfloor \beta n\rfloor$.
Moreover if $\cA\subset\binom{[n]}k$ and $\cB\subset\binom{[n]}l$ are
cross-intersecting families satisfying 
$|\cA||\cB|=\binom{n-1}{k-1}\binom{n-1}{l-1}$ then
$\cA=\{A\in\binom{[n]}k:i\in A\}$ and 
$\cB=\{B\in\binom{[n]}l:i\in B\}$ for some $i\in[n]$.
\end{theorem} 

To relate Theorem~\ref{thm-p} and Theorem~\ref{thm-unif} let us compare
the sizes of $\cA$ and $\cB$ with their measures. We have
\begin{align*}
 |\cA^{(k)}_j|&=\binom{n-1}{k-1}+\binom{n-j-2}{k-j-1},\\
 |\cB^{(l)}_j|&=\binom{n-1}{l-1}-\binom{n-j-2}{l-1}.
\end{align*}
If $\alpha,\beta$, and $j$ are fixed, and $n\to\infty$ with
$k=\alpha n$ and $l=\beta n$, then it follows that
\begin{align*}
 \frac{|\cA^{(k)}_j|}{\binom nk}\to\mu_\alpha(\cA_j)\text{ and }
 \frac{|\cB^{(l)}_j|}{\binom nl}\to\mu_\beta(\cB_j).
\end{align*}
This means that the two conditions 
\[
|\cA^{(k)}_j| |\cB^{(l)}_j|<\binom{n-1}{k-1}\binom{n-1}{l-1} \text{ and }
\mu_\alpha(\cA_j)\mu_\beta(\cB_j)<\alpha\beta 
\]
are corresponding to each other provided 
\begin{align}\label{a=k/n}
\alpha=\frac kn \text{ and } \beta=\frac ln. 
\end{align}

In Theorem~\ref{thm-unif} we assume that $n,k$, and $l$ are large. 
Instead, if we assume a stronger condition than \eqref{boundaryCi} then 
we get another partial solution to Conjecture~\ref{conj1}. To state the
result, let
\[
 \Omega':=\{(k,l)\in\Z^2:k>0,\,l>\frac n2,\,k+l<n\}.
\]

\begin{theorem}\label{thm1}
Let $(k,l)\in\Omega'$. Suppose that
\begin{align}
&\left(1+\frac{n-k}{n-1}\right)\frac{l-1}{n-1}<1,\label{C1}\\
&(n-k)\sum_{i=n-l}^{n-2}\frac1i-(n-l)\sum_{i=k}^{n-2}\frac1i<0.
\label{C2}
\end{align}
Then $M(n,k,l)=\binom{n-1}{k-1}\binom{n-1}{l-1}$. 
Moreover if $\cA\subset\binom{[n]}k$ and $\cB\subset\binom{[n]}l$ are
cross-intersecting families satisfying $|\cA||\cB|=M(n,k,l)$ then
$\cA=\{A\in\binom{[n]}k:i\in A\}$ and 
$\cB=\{B\in\binom{[n]}l:i\in B\}$ for some $i\in[n]$.
\end{theorem}
To compare Theorems~\ref{thm-unif} and \ref{thm1} we note that
if $(\alpha,\beta)\in\Omega$ then $(k,l)\in\Omega'$, where
$k=\lfloor \alpha n\rfloor$ and $l=\lfloor \beta n\rfloor$,
provided $n>n_0(\alpha,\beta)$.
The condition \eqref{C1} is equivalent to 
$|\cA^{(k)}_0||\cB^{(l)}_0|<\binom{n-1}{k-1}\binom{n-1}{l-1}$.
If we assume \eqref{a=k/n}, then \eqref{C1} corresponds to 
\begin{align}\label{C1'}
(2-\alpha)\beta<1, 
\end{align}
which comes from \eqref{boundaryCi}
at $j=0$. Also the condition \eqref{C2} corresponds to
\begin{align}\label{C2'}
 (1-\alpha)\log\frac1{1-\beta}-(1-\beta)\log\frac1{\alpha}<0, 
\end{align}
which is a stronger requirement than \eqref{boundaryCi} for $j\geq 1$.
Let $\Delta'\subset \Omega$ be the set of $(\alpha,\beta)$ satisfying
\eqref{C1'} and \eqref{C2'}. Then $\Delta'\subset\Delta$ and 
$\Delta'$ is illustrated in Figure~\ref{fig3}
filled with light blue.

\begin{figure}[h]
\includegraphics[width=11cm]{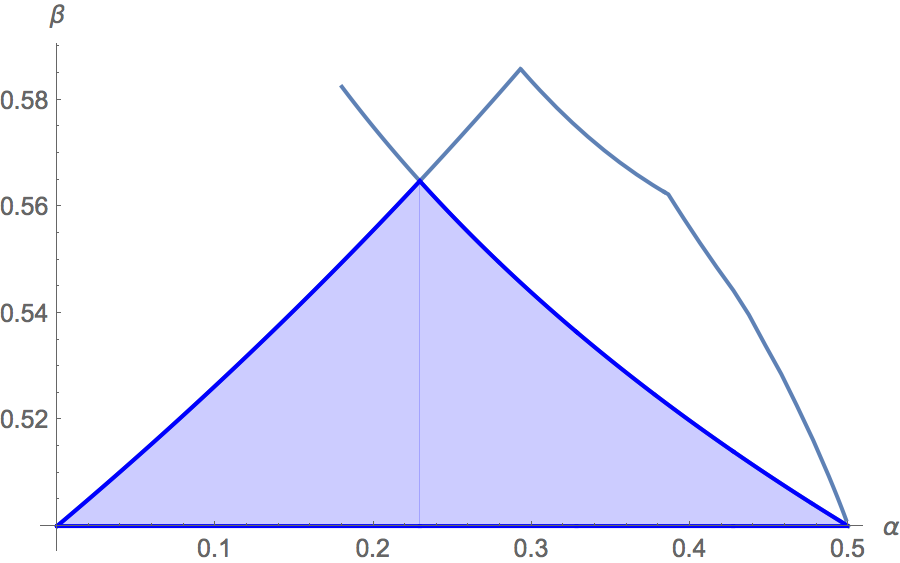}
\caption{The domain $\Delta'$}\label{fig3}
\end{figure} 

The rest of the paper is organized as follows. In Section~\ref{sec2}
we introduce our main tool based on the Kruskal--Katona Theorem.
In Section~\ref{sec3} we prove the first part of Theorem~\ref{thm-unif}
(the determination of $M(n,k,l)$), and in Section~\ref{sec4} we prove
the second part of Theorem~\ref{thm-unif} (the uniqueness of the extremal
structure). Then, in Section~\ref{sec5}, we prove Theorem~\ref{thm-p}
using Theorem~\ref{thm-unif}. In the last section we outline the proof
of Theorem~\ref{thm1}.

\section{Preliminaries}\label{sec2}
In this section we state a consequence of the Kruskal--Katona Theorem,
which is one of the main tools for the proof of our results.
For a real number $x$ and an integer $t$ with $x\geq t>0$ define
$\binom xt:=\frac{x(x-1)\cdots(x-t+1)}{t!}$. We also define
$\binom x0=1$ and $\binom x{t}=0$ if $x<t$. 
For given positive integers $m$ and $u$ we can represent $m$ in the form as
\begin{align}\label{u-cascade of m}
 m=\binom {a_u}u+\binom {a_{u-1}}{u-1}+\cdots+\binom {a_{u-t}}{u-t} 
\end{align} 
with $a_u>a_{u-1}>\cdots>a_{u-t}\geq u-t\geq 1$, 
where $t,a_u,\ldots,a_{u-t}$ are all integers. 
The representation is unique, and it is called the $u$-cascade form of $m$. 
For a family $\cF\subset\binom{[n]}u$ and an integer $v$ with $0<v<u$,
let us define the $v$-shadow of $\cF$ by 
\[
\sigma_v(\cF):=\{G\in\binom{[n]}v:G\subset F\text{ for some }F\in\cF\}. 
\]
If $|\cF|=m$ and its $u$-cascade form is given by \eqref{u-cascade of m},
then it follows from the Kruskal--Katona Theorem \cite{Kr, Ka} that
\[
 |\sigma_v(\cF)|\geq
\binom {a_u}v+\binom {a_{u-1}}{v-1}+\cdots+\binom {a_{u-t}}{v-t}.
\]
If, moreover, we choose an integer $s$ and a real number $x$ 
with $0<s<t$ and $a_{u-s-1}\leq x<a_{u-s}$ such that
\begin{align}\label{ux-cascade}
 m=\binom {a_u}u+\binom {a_{u-1}}{u-1}+\cdots+\binom {a_{u-s}}{u-s} 
+\binom{x}{u-s-1}, 
\end{align}
in other words, if we choose $s$ and $x$ so that
\[
 \binom {a_{u-s-1}}{u-s-1}+\cdots+\binom {a_{u-t}}{u-t} =\binom x{u-s-1},
\]
then it follows from the Lov\'asz version of the Kruskal--Katona Theorem 
\cite{Lo} that
\begin{align}\label{v-cascade}
 |\sigma_v(\cF)|\geq
\binom {a_u}v+\binom {a_{u-1}}{v-1}+\cdots+\binom {a_{u-s}}{v-s}
+\binom {x}{v-s-1}.
\end{align}
Katona also proved that if $|\cF|=\binom{a_u}u$ then 
$|\sigma_v(\cF)|=\binom{a_u}v$ holds if and only if $\cF\cong\binom{[a_u]}u$,
and we will use this fact in Section~\ref{sec4}. For more about
the structure of minimum shadows, see \cite{FG, Mors}. 

If $\cF\subset\binom{[n]}u$ and $\cG\subset\binom{[n]}v$ are cross-intersecting,
then $\sigma_v(\cF^c)\cap\cG=\emptyset$, where
$\cF^c:=\{[n]\setminus F:F\in\cF\}$. 
Thus $|\cG|\leq\binom nv-|\sigma_v(\cF^c)|$.
Noting that $|\cF|=|\cF^c|$ we have the following lemma, 
which we will use repeatedly in the proof of Theorem~\ref{thm-unif}.
\begin{lemma}\label{KK lemma}
Let $\cF\subset\binom{[n]}u$ and $\cG\subset\binom{[n]}v$ be cross-intersecting.
If $m=|\cF|$ is represented by \eqref{ux-cascade}, then
$|\cG|\leq\binom nv-(\text{the RHS of \eqref{v-cascade}})$.
\end{lemma}

Let $(k,l)\in\Omega'$, and let
$\cA\subset\binom{[n]}k$ and $\cB\subset\binom{[n]}l$ be cross-intersecting.
To prove $|\cA||\cB|\leq\binom{n-1}{k-1}\binom{n-1}{l-1}$ we may assume that
$|\cA|>\binom{n-1}{k-1}$ or $|\cB|>\binom{n-1}{l-1}$.
We can settle the latter case easily.

\begin{claim}\label{easy case}
If $|\cB|>\binom{n-1}{l-1}$ then $|\cA||\cB|<\binom{n-1}{k-1}\binom{n-1}{l-1}$.
\end{claim}
\begin{proof}
Let $|\cB|=\binom{n-1}{n-l}+\binom x{n-l-1}$ for $n-l-1\leq x\leq n-1$. Then,
by Lemma~\ref{KK lemma},
$|\cA|\leq\binom nk-\binom{n-1}k-\binom x{k-1}=\binom{n-1}{k-1}-\binom x{k-1}$. 
To prove $|\cA||\cB|<\binom{n-1}{k-1}\binom{n-1}{l-1}$ it suffices to show that
$\binom{n-1}{l-1}\binom x{k-1}>\binom{n-1}{k-1}\binom x{n-l-1}$.
This is equivalent to
\[
 (n-k)\cdots l>(x-k+1)\cdots(x-n+l+2)(n-l).
\]
Both sides consist of $n-k-l+1$ products, and by comparing the corresponding
terms one by one we see that $n-k-i+1\geq x-k-i+2$ for $1\leq i\leq n-l-k$, 
and $l>n-l$.
\end{proof}

Note that we did not use any of \eqref{boundaryCi}, \eqref{C1'}, or 
\eqref{C2'} to prove Claim~\ref{easy case}. Thus Claim~\ref{easy case}
is valid for both Theorems~\ref{thm-unif} and \ref{thm1}.

\section{Proof of Theorem~\ref{thm-unif}: The inequality}\label{sec3}
In this section we prove $M(n,k,l)=\binom{n-1}{k-1}\binom{n-1}{l-1}$
under the premises of Theorem~\ref{thm-unif}. In short, we just apply
Lemma~\ref{KK lemma} several times according to the size of $\cA$, and
compute the corresponding upper bounds. 
However this task is rather involved because we need to choose the interval 
for the size of $\cA$ carefully to obtain the right bound.

Fix $(\alpha,\beta)\in\Delta$.
Recall that $\alpha$ and $\beta$ satisfy \eqref{boundaryCi} 
for all $j\geq 0$, so we have \eqref{b<ei}.
For every fixed $0<\alpha<1/2$ we have $e_j\to 1-\alpha$ as $j\to\infty$. 
Solving $e_0=e_1$ for $\alpha$ we get 
$\alpha=1-\frac1{\sqrt{2}}=:\tilde\alpha$. In this case 
$e_0=e_1=2\tilde\alpha=:\tilde\beta$.
Thus, by \eqref{b<ei}, we have
\begin{align}\label{beta<tilde beta}
\beta\leq\min\{e_0,e_1\}\leq\tilde\beta.   
\end{align}
Indeed there is a cusp at $(\alpha,\beta)=(\tilde\alpha,\tilde\beta)$ 
in Figure~\ref{fig2}. Let $\aa:=1-\alpha$ and $\bb:=1-\beta$.

We will often use the following fact without mentioning of it.
\begin{claim}\label{claim:gamma}
Let $k,l,n\in\N$ with
$k=\lfloor\alpha n\rfloor$, $l=\lfloor\beta n\rfloor$.
For fixed $0<t\leq s$ with $t\in\N, s\in\R$ it follows that
\[
 \binom{n-s}{n-k-t}\bigg/\binom nk\to\alpha^{s-t}\aa^{t},\qquad
 \binom{n-s}{l-t}\bigg/\binom nl\to\beta^t\bb^{s-t}
\]
as $n\to\infty$. 
\end{claim}
\begin{proof}
Recall the following basic properties of the Gamma function:
\[
 \binom xn=\frac{\Gamma(x+1)}{\Gamma(n+1)\Gamma(x-n+1)}
\]
for $x\in\R$ with $x\geq n$, and
\[
 \lim_{n\to\infty}\frac{\Gamma(y+n)}{n^y\Gamma(n)}=1
\]
for $y\in\R$.
Thus we have
\begin{align*}
\lim_{n\to\infty} \binom{n-s}{n-k-t}\bigg/\binom nk&=
\lim_{n\to\infty}
\frac{\Gamma(n-s+1)}{\Gamma(n+1)}
\frac{\Gamma(k+1)}{\Gamma(k-s+t+1)}
\frac{\Gamma(n-k+1)}{\Gamma(n-k-t+1)}\\
&=\lim_{n\to\infty}n^{-s} k^{s-t} (n-k)^t= \alpha^{s-t}\aa^t.
\end{align*} 
One can prove the second item similarly.
\end{proof}

The goal of this section is to show the following lemma.
\begin{lemma}\label{main lemma}
Let $k,l,n$ satisfy the premises of Theorem~\ref{thm-unif}.
Suppose that $\cA\subset\binom{n}k$ and $\cB\subset\binom{[n]}l$ are
cross-intersecting. Then $|\cA||\cB|<\binom{n-1}{k-1}\binom{n-1}{l-1}$
unless $|\cA|=\binom{n-1}{k-1}$ and $|\cB|=\binom{n-1}{l-1}$.
\end{lemma}

The case $|\cB|>\binom{n-1}{l-1}$ is shown in Claim~\ref{easy case}.
So from now on we assume that $|\cA|>\binom{n-1}{k-1}$.
We will apply Lemma~\ref{KK lemma} with $\cF=\cA$ and $\cG=\cB$
to show that the upper bound for $|\cA||\cB|$ from the lemma is less than
$\binom{n-1}{k-1}\binom{n-1}{l-1}$.
We break the proof into several cases according to the size of $\cA$.
To this end we introduce a constant $i_0$ based on the following fact.
\begin{claim}\label{claim:i0}
There exists a constant $i_0=i_0(\alpha)$ such that
\begin{align}\label{ineq:large i}
(1+\alpha^{i-2}\aa)\log\frac1{1-e_{i-2}}<\log\frac1{\alpha}
\end{align}
holds for all $i\geq i_0$.
\end{claim}
\begin{proof}
By differentiating the LHS of \eqref{ineq:large i} with respect to $i$ we get 
$\frac1{\alpha^2 (i-1)^2}h(i)$, where 
\[
\lim_{i\to\infty}h(i)=\alpha^2\log(\alpha^{-1}-1)>0
\]
for $0<\alpha<\frac12$. Thus the LHS of \eqref{ineq:large i} is increasing in
$i$ provided $i$ sufficiently large. Moreover, using $e_i\to 1-\alpha$ as
$i\to\infty$, it follows that the LHS of \eqref{ineq:large i} 
goes to $\log\frac1{\alpha}$ as $i\to\infty$. This means that 
\eqref{ineq:large i} holds if $i$ is large enough.
\end{proof}
We fix an integer $i_0$ from Claim~\ref{claim:i0}. Note that the $i_0$ is
independent of $n$.
Then we consider the following four cases separately:
\begin{itemize}
\item $\binom{n-1}{n-k}<|\cA|\leq\binom{n-1}{n-k}+\binom{n-i_0}{n-k-1}$
\item $\binom{n-1}{n-k}+\binom{n-i_0}{n-k-1}<|\cA|\leq
\binom{n-1}{n-k}+\binom{n-3}{n-k-1}$.
\item $\binom{n-1}{n-k}+\binom{n-3}{n-k-1}<|\cA|\leq
\binom{n-1}{n-k}+\binom{n-2}{n-k-1}$.
\item $\binom{n-1}{n-k}+\binom{n-2}{n-k-1}<|\cA|$.
\end{itemize}

We will often use the following easy observation.

\begin{lemma}\label{lemma:limit}
Let $0<\alpha<1$ be a fixed real number, and
let $n,k$ be positive integers with $k=\lfloor \alpha n\rfloor$.
Let $\{\cF_n\}$ be a sequence of families with $\cF_n\subset\binom{[n]}k$,
and let $\{c_n\}$ be a sequence of real numbers with $0<c_n<1$. Suppose that
$ \lim_{n\to\infty}{|\cF_n|}/{\binom nk}<\lim_{n\to\infty} c_n$.
Then there exists an $n_0$ such that 
${|\cF_n|}/{\binom nk}< c_n$ for all $n>n_0$.
\end{lemma}

Now we prove several inequalities (Claims~\ref{claimA}--\ref{claimC-app})
which we need for the proof of Lemma~\ref{main lemma}. We mention that we 
could prove Claims~\ref{claimA}, \ref{claimB}, and \ref{claimC} in a unified 
way, but to make the description simpler we prove them separately.
We also mention that the main idea of the proof of these claims is
taken from \cite{MT}.
After proving these claims Lemma~\ref{main lemma} will follow easily.

Here we recall our assumption throughout this section:
$(\alpha,\beta)\in\Delta$, $k=\lfloor\alpha n\rfloor$, and
$l=\lfloor\beta n\rfloor$.

\begin{claim}\label{claimA}
Let $i\geq 2$ and $\epsilon\geq 0$ be fixed integers.
Let $X=\binom{n-1}{n-k}+\binom{n-i}{n-k-1}$,
$Y=\binom{n-1}{l-1}-\binom{n-i}{l-1}$, and define a polynomial $F(x)$ by
\[
 F(x):=\left(X+\binom x{n-k-2}\right)\left(Y-\binom x{l-2}\right).
\] 
Let $M:=\max\{XY,F(n-i-\epsilon)\}$.
Then there exists $n_0=n_0(i,\epsilon)$ such that for all $n>n_0$ the following
holds: if 
\begin{align}\label{cond2}
\frac{1-\bb^{i-1}}{\beta\bb^{i-2+\epsilon}}\log\frac1{\bb}\leq
\frac{1+\alpha^{i-2}\aa}{\alpha^{i-3+\epsilon}\aa^2}\log\frac1{\alpha},
\end{align}
then 
\begin{align}\label{ineq:A}
F(x)<M
\end{align}
for all $x$ with $n-k-3<x<n-i-\epsilon$.
\end{claim}

\begin{proof}
The inequality \eqref{ineq:A} is rewritten as
\[
\frac{F(x)-XY}{\binom x{l-2}}=
-X+Y\frac{\binom x{n-k-2}}{\binom x{l-2}}-\tbinom x{n-k-2}<
\frac{M-XY}{\binom x{l-2}}.  
\]
This is equivalent to $f(x)<(M-XY)\frac{(n-k-2)!}{\binom {x}{l-2}}$,
where $f(x)$ is defined by
\begin{align}
f(x)&:=(F(x)-XY)\frac{(n-k-2)!}{\binom x{l-2}}\nonumber\\ 
&=-(n-k-2)!\,X+(l-2)!\,Y\prod_{j=l-2}^{n-k-3}(x-j)-\prod_{j=0}^{n-k-3}(x-j). 
\label{def:f}
\end{align}
We actually prove a slightly stronger inequality
\begin{align}\label{def2:f}
 f(x)<(M-XY)\frac{(n-k-2)!}{\binom {n-i-\epsilon}{l-2}}=:m 
\end{align}
for $n-k-3<x<n-i-\epsilon$. Clearly it is
true at the two ends, indeed we have $f(n-k-3)=-(n-k-2)!X<0\leq m$, 
and $f(n-i-\epsilon)\leq m$ follows from the definition of $f(x)$ with 
$F(n-i-\epsilon)\leq M$.

We prove \eqref{def2:f} by contradiction. Suppose that there is some 
$y$ with $n-k-3<y<n-i-\epsilon$ such that
$f(y)>m$. Since $f(x)\leq m$ at the two ends we may assume that 
$\frac{d}{dy}f(y)=0$. This yields
\[
 \prod_{j=0}^{n-k-3}(y-j)=(l-2)!\,Y\prod_{j=l-2}^{n-k-3}(y-j)
\sum_{j=l-2}^{n-k-3}\frac1{y-j}\bigg/\left(\sum_{j=0}^{n-k-3}\frac1{y-j}\right).
\]
Substituting the RHS into the last term in \eqref{def:f} we get a new inequality
\begin{align*}
m< g(y):=
-(n-k-2)!\,X+(l-2)!\,Y\prod_{j=l-2}^{n-k-3}(y-j)
\sum_{j=0}^{l-3}\frac1{y-j}\bigg/\left(\sum_{j=0}^{n-k-3}\frac1{y-j}\right).
\end{align*} 
Since
\[
(y-n+k+3)\sum_{j=0}^{l-3}\frac1{y-j}=\frac{y-n+k+3}{y}+
\cdots+\frac{y-n+k+3}{y-l+3}
\]
is increasing in $y$, $g(y)$ is also increasing in $y$. 
So we must have $g(n-i-\epsilon)>m$.
We will show that this cannot happen.
Using
\begin{align*}
&\lim_{n\to\infty}{(n-k-2)!\,X}\,\frac{k!\, n^2}{n!}
=\frac{\alpha+\alpha^{i-1}\aa}{\aa^2},\\
&\lim_{n\to\infty}(l-2)!\,Y\prod_{j=l-2}^{n-k-3}(y-j)\,\frac{k!\, n^2}{n!}
=\frac{\beta-\beta\bb^{i-1}}{\beta^2}\left(\frac\alpha{\bb}\right)^{i+\epsilon-2},\\
&\lim_{n\to\infty}\sum_{j=0}^{l-3}\frac1{(n-i-\epsilon)-j}=
\lim_{n\to\infty}\int_{n-l-i+3-\epsilon}^{n-i+1-\epsilon}\frac1ydy=\log\frac1\bb,\\
&\lim_{n\to\infty}\sum_{j=0}^{n-k-3}\frac1{(n-i-\epsilon)-j}=
\lim_{n\to\infty}\int_{k-i+3-\epsilon}^{n-i+1-\epsilon}\frac1ydy=\log\frac1\alpha,
\end{align*}
it follows from $g(n-i-\epsilon)>m>0$ that
\begin{align*}
-\frac{\alpha+\alpha^{i-1}\aa}{\aa^2}+
\frac{\beta-\beta\bb^{i-1}}{\beta^2}\left(\frac\alpha{\bb}\right)^{i+\epsilon-2}
\log\frac1\bb\bigg/\log\frac1\alpha>0.
\end{align*}
However this reduces to the opposite inequality to \eqref{cond2},
a contradiction.
\end{proof}

\begin{claim}\label{claimA-app}
Let $i,\epsilon, X,Y,F(x)$ and $M$ be as in Claim~\ref{claimA}.
If one of the following holds, 
\begin{itemize}
\item[(i)] $i=2,3$ and $\epsilon=1$, 
\item[(ii)] $i=3$, $\epsilon=0$, and $\alpha<0.27$,
\item[(iii)] $i\geq 4$ and $\epsilon=0$, 
\end{itemize}
then for sufficiently large $n$ it follows
$F(x)<\binom{n-1}{k-1}\binom{n-1}{l-1}$ for $n-k-3<x<n-i-\epsilon$.
\end{claim}

\begin{proof}
We use Claim~\ref{claimA}. To this end
we show \eqref{cond2} and $M<\binom{n-1}{k-1}\binom{n-1}{l-1}$ in each case. 
Then the result will follow from \eqref{ineq:A}.
In all cases $XY<\binom{n-1}{k-1}\binom{n-1}{l-1}$ for $n$ sufficiently large.
Indeed we have $\lim_{n\to\infty}\frac{XY}{\binom nk\binom nl}=
(\alpha+\alpha^{i-1}\aa)(\beta-\beta\bb^{i-1})<\alpha\beta$ 
by \eqref{boundaryCi} at $j=i-2$, and this together 
with Lemma~\ref{lemma:limit} gives us the upper bound for $XY$. Thus, 
to prove $M<\binom{n-1}{k-1}\binom{n-1}{l-1}$, we only need to show
$F(n-i-\epsilon)<\binom{n-1}{k-1}\binom{n-1}{l-1}$.
This is true if $\epsilon=0$. Indeed, in this case, we have 
$\lim_{n\to\infty}F(n-i)/(\binom nk\binom nl)=
(\alpha+\alpha^{i}\aa)(\beta-\beta\bb^{i})$ and this is less than $\alpha\beta$
by \eqref{boundaryCi} at $j=i-1$. 

(i) Let $i=2,3$, and $\epsilon=1$. 
To apply Claim~\ref{claimA}
we show that $F(n-i-1)<\binom{n-1}{k-1}\binom{n-1}{l-1}$
and \eqref{cond2}.
For the former we note that
\[
\lim_{n\to\infty}\frac{F(n-i-1)} {\binom nk\binom nl}=
(\alpha+\alpha^{i-1}\aa+\alpha^{i-1}\aa^2)
(\beta-\beta\bb^{i-1}-\beta^2\bb^{i-1}),
\]
and in view of Lemma~\ref{lemma:limit} we need to show that
\begin{align}\label{eq:F(n-i-1)}
(1+\alpha^{i-2}\aa+\alpha^{i-2}\aa^2)
(1-\bb^{i-1}-\beta\bb^{i-1})<1.
\end{align}
The LHS is increasing in $\beta$. By \eqref{b<ei} 
it suffices to show the inequality
at $\beta=e_0$ if $i=2$, and at $\beta=\min\{e_0,e_1\}$ if $i=3$.
In these cases we can verify \eqref{eq:F(n-i-1)} by direct computation.

Next we show \eqref{cond2}, that is,
\begin{align}\label{app:cA}
\frac{1-\bb^{i-1}}{\beta\bb^{i-1}} \log\frac1{\bb}
\leq\frac{1+\alpha^{i-2}\aa}{\alpha^{i-2}\aa^2}\log\frac1{\alpha}.
\end{align}
The LHS is increasing in $\beta$, indeed both $(1-\bb^{i-1})/\bb^{i-1}$
and $\frac1\beta\log\frac1{\bb}$ are increasing in $\beta$. 
Thus it suffices to check the inequality
at $\beta=\min\{\aa,e_{i-2}\}$.

If $\beta=\aa$ then \eqref{app:cA} reduces to $2\alpha^{i-1}\aa\geq 1-2\alpha$. 
This holds for $\alpha_0<\alpha<1/2$, where
$\alpha_0\in(0,1/2)$ is a unique root of $2(1-x)x^{i-1}=1-2x$. 
Thus \eqref{app:cA} holds for $\alpha_0<\alpha<1/2$.

If $\beta=e_{i-2}$ then, using 
$\bb^{i-1}=\frac{\alpha^{i-2}\aa}{1+\alpha^{i-2}\aa}$ from \eqref{eq:ej},
\eqref{app:cA} reduces to 
\begin{align}\label{app:cA2}
\frac1\beta\log\frac1\bb\bigg/\frac{1}{\aa}\log\frac1\alpha\leq
1+\alpha^{i-2}\aa. 
\end{align}
Since $\frac1{1-x}\log\frac1x$ is decreasing in $x$, the LHS of \eqref{app:cA2}
is less than $1$ if $\bb=1-e_{i-2}>\alpha$, that is, 
$\alpha^{i-1}\aa<1-2\alpha$. This holds for $0<\alpha<\alpha_1$, where
$\alpha_1\in(0,1/2)$ is a root of $(1-x)x^{i-1}=1-2x$. 
Thus \eqref{app:cA} also holds for $0<\alpha<\alpha_1$.

By definition $\alpha_0<\alpha_1$ follows. Therefore \eqref{app:cA} holds
for all $0<\alpha<1/2$. Consequently we have 
$M<\binom{n-1}{k-1}\binom{n-1}{l-1}$ and \eqref{cond2}. Thus, 
by Claim~\ref{claimA}, we get $F(x)<\binom{n-1}{k-1}\binom{n-1}{l-1}$.

(ii) Let $i=3$ and $\epsilon=0$. 
In this case we only need to show \eqref{cond2}, that is,
\[
 \frac{1-\bb^2}{\beta \bb}\log\frac1{\bb}\leq
\frac{1+ \alpha\aa}{\aa^2}\log\frac1\alpha.
\]
The LHS is increasing in $\beta$, so it suffices to check the inequality
at $\beta=e_0$. Indeed this is true if $\alpha<0.27$. 

(iii)
Let $i\geq 4$ and $\epsilon=0$. We prove \eqref{cond2}, that is,
\[
 z:=\frac{\alpha^{i-3}\aa^2(1-\bb^{i-1})\log\frac1{\bb}}
{\beta\bb^{i-2}(1+\alpha^{i-2}\aa)\log\frac1\alpha}\leq 1.
\]
The LHS is increasing in $\beta$, and it suffices to show the inequality at 
$\beta=\min\{e_{i-3},e_{i-2}\}$. 

By solving $e_{i-2}=e_{i-3}$ for $\alpha$ let $\alpha_*(i)$ be the solution
in $(0,1)$. For example, $\alpha_*(4)=0.386\ldots$, and
it follows that $\frac1e<\alpha_*(4)<\alpha_*(5)<\cdots$. 
Let $\beta_*(i)$ be defined by $\beta_*(i)=e_{i-2}=e_{i-3}$ 
at $\alpha=\alpha_*(i)$, e.g., $\beta_*(4)=0.562\ldots$, 
see Figure~\ref{fig1}. Then we have
\[
 \min\{e_{i-3},e_{i-2}\}=
\begin{cases}
 e_{i-3} & \text{if }0<\alpha\leq \alpha_*(i),\\
 e_{i-2} & \text{if }\alpha_*(i)\leq\alpha<\frac12.
\end{cases}
\]
We also note that both $e_{i-3}$ and $e_{i-2}$ are decreasing in $\alpha$.
(Here we need $i\geq 4$, because $e_0$ is increasing in $\alpha$.)

If $\beta=e_{i-3}$ then, 
using $\bb^{i-2}=\frac{\alpha^{i-3}\aa}{1+\alpha^{i-3}\aa}$, it follows that
\begin{align}\label{zi-3}
 z(\beta=e_{i-3})=\frac{(1+\beta\alpha^{i-3}\aa)\aa\log\frac1\bb}
{\beta(1+\alpha^{i-2}\aa)\log\frac1{\alpha}}.
\end{align}
Then, using $\alpha^{i-3}\aa=\frac{\bb^{i-2}}{1-\bb^{i-2}}$, 
$z(\beta=e_{i-3})\leq 1$ reduces to
\[
 \frac{1-\bb^{i-1}}{\beta(1-\bb^{i-2})}\log\frac1\bb
\leq\frac{1+\alpha^{i-2}\aa}{\aa}\log\frac1\alpha.
\]
We need to check this inequality for $0<\alpha\leq\alpha_*(i)$.
Since the LHS is increasing in $\beta$ and the RHS is decreasing in $\alpha$,
it suffices to check the inequality at 
$(\alpha,\beta)=(\alpha_*(i),\beta_*(i))$.

If $\beta=e_{i-2}$ then, using
$\bb^{i-1}=\frac{\alpha^{i-2}\aa}{1+\alpha^{i-2}\aa}$, it follows that
\begin{align}\label{zi-2}
 z(\beta=e_{i-2})=\frac{\aa\bb\log\frac1\bb}
{\alpha\beta(1+\alpha^{i-2}\aa)\log\frac1{\alpha}},
\end{align}
and $z(\beta=e_{i-2})\leq 1$ is equivalent to
\[
\frac\bb\beta\log\frac1\bb
\leq\frac{\alpha(1+\alpha^{i-2}\aa)}{\aa}\log\frac1\alpha.
\]
We need to check this inequality for $\alpha_*(i)\leq\alpha<\frac12$.
In this case the LHS is decreasing in $\beta$ and the RHS is increasing in 
$\alpha$, so again it suffices to check the inequality at 
$(\alpha,\beta)=(\alpha_*(i),\beta_*(i))$.

Now we consider the case $(\alpha,\beta)=(\alpha_*(i),\beta_*(i))$.
For simplicity let us just write $\alpha_*$ and $\beta_*$ omitting $i$,
and let $\aa_*=1-\aa_*$ and $\bb_*=1-\bb_*$.
Then, by \eqref{zi-3} and \eqref{zi-2}, 
$z(\beta=e_{i-3})=z(\beta=e_{i-2})$ implies that
$\beta_*=\frac{\aa_*}{1+\alpha_*^{i-2}\aa_*}$, and substituting this 
into \eqref{zi-2} we have
\[
 z_*:=z(\beta=\beta_*)=\frac{\bb_*\log\frac1{\bb_*}}{\alpha_*\log\frac1{\alpha_*}}.
\]
Recall that we always assume $\alpha+\beta<1$, and in particular, 
$\frac1e<\alpha_*<\bb_*$.
(Here we need $i\geq 4$ because $\alpha_*(3)<\frac1e$.)
Since $x\log\frac1x$ is decreasing for $\frac1e<x$ we get $z_*<1$, as required.
\end{proof}

\begin{claim}\label{claimB}
Let $i\geq 2$ and $\epsilon\geq 1$ be fixed integers.
Let $X=\binom{n-1}{n-k}+\binom{n-i}{n-k-1}+\binom{n-i-1}{n-k-2}$ and
$Y=\binom{n-1}{l-1}-\binom{n-i}{l-1}-\binom{n-i-1}{l-2}$.
Define a polynomial $F(x)$ by
\[
F(x):=\left(X+\binom x{n-k-3}\right)\left(Y-\binom x{l-3}\right). 
\]
Let $M:=\max\{XY,F(n-i-\epsilon)\}$.
Then there exists $n_0=n_0(i,\epsilon)$ such that for all $n>n_0$ the following
holds: if 
\begin{align}\label{ineq:claimB}
\frac{1-\bb^{i-1}-\beta\bb^{i-1}}{\beta^2\bb^{i-3+\epsilon}}\log\frac1\bb
\leq\frac{1+\aa\alpha^{i-2}+\aa^2\alpha^{i-2}}{\aa^3\alpha^{i-4+\epsilon}}
\log\frac1\alpha,
\end{align}
then 
\begin{align*}
F(x)<M
\end{align*}
for all $x$ with $n-k-4<x<n-i-\epsilon$.
\end{claim}

\begin{proof}
The proof is almost identical to the proof of Claim~\ref{claimA},
and we only include a sketch. Let 
\begin{align*}
f(x)&:=(F(x)-XY)\frac{(n-k-3)!}{\binom x{l-3}}\\
&=-(n-k-3)!\,X+(l-3)!\,Y\prod_{j=l-3}^{n-k-4}(x-j)-\prod_{j=0}^{n-k-4}(x-j).
\end{align*}
Then $F(x)<M$ follows from
\[
 f(x)<(M-XY)\frac{(n-k-3)!}{\binom {n-i-\epsilon}{l-3}}=:m.  
\]
Suppose, to the contrary, that there is some $y$ such that $f(y)\geq m$.
Then we may assume that
\[
m< g(y):=
-(n-k-3)!\,X+(l-3)!\,Y\prod_{j=l-3}^{n-k-4}(y-j)
\sum_{j=0}^{l-4}\frac1{y-j}\bigg/\left(\sum_{j=0}^{n-k-4}\frac1{y-j}\right).
\]
Since $g(y)$ is increasing in $y$, $g(n-i-\epsilon)>m>0$ must hold.
However, considering $n\to\infty$, $g(n-i-\epsilon)>0$ yields the opposite
inequality to \eqref{ineq:claimB}, a contradiction.
\end{proof}

\begin{claim}\label{claimB-app}
Let $i,\epsilon, X,Y,F(x)$ and $M$ be as in Claim~\ref{claimB}.
If either
\begin{itemize}
 \item[(i)] $i=2$ and $\epsilon=2$, or
\item[(ii)] $i=3$, $\epsilon=1$, and $\alpha>0.23$,
\end{itemize}
then for sufficiently large $n$ it follows
$F(x)<\binom{n-1}{k-1}\binom{n-1}{l-1}$ for $n-k-4<x<n-i-\epsilon$.
\end{claim}

\begin{proof}
(i) 
First we show that $M<\binom{n-1}{k-1}\binom{n-1}{l-1}$.
To bound $XY$ we note that
\[
\lim_{n\to\infty}\frac{XY}{\binom nk\binom nl}
=(\alpha+\alpha\aa+\alpha\aa^2)(\beta-\beta\bb-\beta^2\bb).
\]
In view of Lemma~\ref{lemma:limit} we need to show that
$(1+\aa+\aa^2)(1-\bb-\beta\bb)<1$. The LHS is increasing in $\beta$, and
it suffices to check the inequality at $\beta=e_0$, that is,
$\frac{3-3\alpha+\alpha^2}{(2-\alpha)^2}<1$. This is indeed true for
$0<\alpha<1/2$. Similarly as for $F(n-4)$ it suffices to check
$(1+\aa+\aa^2+\aa^3)(1-\bb-\beta\bb-\beta^2\bb)<1$ at $\beta=e_0$,
and it is also true. 

Next we check \eqref{ineq:claimB}, that is,
\[
 \frac{1-\bb-\beta\bb}{\beta^2\bb}\log\frac1\bb
\leq\frac{1+\aa+\aa^2}{\aa^3}\log\frac1\alpha.
\]
Since the LHS is increasing in $\beta$ it suffices to check the
inequality at $\beta=e_0$, which can be verified by direct computation. 

(ii) In this case $M<\binom{n-1}{k-1}\binom{n-1}{l-1}$ follows by showing
$(1+\alpha\aa+\alpha\aa^2)(1-\bb^2-\beta\bb^2)<1$ and
$(1+\aa)(1-\bb)<1$. The latter is equivalent to \eqref{boundaryCi} for $j=0$.
The LHS of the former is increasing in $\beta$ and the inequality is
indeed true at $\beta=\min\{e_0,e_1\}$.
As for \eqref{ineq:claimB} we need to show
\[
 \frac{1-\bb^2-\beta\bb^2}{\beta^2\bb}\log\frac1\bb
\leq\frac{1+a \aa+a \aa^2}{\aa^3}\log\frac1\alpha.
\]
The LHS is increasing in $\beta$, and one can verify the inequality
at $\beta=e_0$ for $\alpha>0.23$ in this case.
\end{proof}

\begin{claim}\label{claimC}
Let $i\geq 2$ be a fixed integer.
Let $X=\binom{n-1}{n-k}$ and $Y=\binom{n-1}{l-1}$.
Define a polynomial $F(x)$ by
\[
F(x):=\left(X+\binom x{n-k-1}\right)\left(Y-\binom x{l-1}\right). 
\]
Then there exists $n_0=n_0(i)$ such that for all $n>n_0$ the following
holds: if 
\begin{align}\label{ineq:claimC}
\frac1{\bb^{i-1}}\log\frac1\bb<\frac1{\alpha^{i-2}\aa}\log\frac1\alpha,
\end{align}
then 
\begin{align*}
F(x)<XY
\end{align*}
for all $x$ with $n-k-2<x<n-i$.
\end{claim}

\begin{proof}
The proof is similar to and easier than that of Claims~\ref{claimA}
and \ref{claimB}. 
Note that $F(n-i)<XY$ follows from \eqref{boundaryCi} at $j=i-2$. Let 
\begin{align*}
f(x)&:=(F(x)-XY)\frac{(n-k-1)!}{\binom x{l-1}}\\
&=-(n-k-1)!\,X+(l-1)!\,Y\prod_{j=l-1}^{n-k-2}(x-j)-\prod_{j=0}^{n-k-2}(x-j).
\end{align*}
Then $F(x)<XY$ follows from $f(x)<0$.
Suppose, to the contrary, that there is some $y$ such that $f(y)\geq 0$.
Then we may assume that
\[
0\leq g(y):=
-(n-k-1)!\,X+(l-1)!\,Y\prod_{j=l-1}^{n-k-2}(y-j)
\sum_{j=0}^{l-2}\frac1{y-j}\bigg/\left(\sum_{j=0}^{n-k-2}\frac1{y-j}\right).
\]
Since $g(y)$ is increasing in $y$, $g(n-i)\geq 0$ must hold, that is,
\begin{align}\label{cond:claimC}
 \frac{(k-1)!(n-k)}{(n-l)!}\prod_{j=l-1}^{n-k-2}(n-i-j)
\sum_{j=0}^{l-2}\frac1{n-i-j}\geq\sum_{j=0}^{n-k-2}\frac1{n-i-j}.
\end{align}
But this reduces to the opposite inequality to \eqref{ineq:claimC}
by considering $n\to\infty$.
\end{proof}

\begin{claim}\label{claimC-app}
Let $i,X,Y$, and $F(x)$ be as in Claim~\ref{claimC}.
If $i=i_0$, that is,
\begin{align}\label{ineq:cC-app}
(1+\alpha^{i-2}\aa)\log\frac1{1-e_{i-2}}<\log\frac1{\alpha}, 
\end{align}
then for sufficiently large $n$ it follows
$F(x)<\binom{n-1}{k-1}\binom{n-1}{l-1}$ for $n-k-2<x<n-i$.
\end{claim}

\begin{proof}
Since $XY=\binom{n-1}{k-1}\binom{n-1}{l-1}$ we only need to show
\eqref{ineq:claimC}. The LHS of \eqref{ineq:claimC} is increasing in $\beta$,
it suffices to check the inequality at $\beta=e_{i-2}$. In this case,
using $\bb^{i-1}=\frac{\alpha^{i-2}\aa}{1+\alpha^{i-2}\aa}$, the
inequality \eqref{ineq:claimC} reduces to \eqref{ineq:cC-app}.
\end{proof}

Using the above claims we are going to prove
\begin{align}\label{AB<EKR}
 |\cA||\cB|<\binom{n-1}{k-1}\binom{n-1}{l-1}.
\end{align}  

\subsection{The case $\binom{n-1}{n-k}<|\cA|\leq\binom{n-1}{n-k}+\binom{n-i}{n-k-1}$ for $i=i_0$}
Let $|\cA|=\binom{n-1}{n-k}+\binom x{n-k-1}$ for $n-k-1\leq x\leq n-i$.
Then, by Lemma~\ref{KK lemma}, $|\cB|\leq\binom{n-1}{l-1}-\binom x{l-1}$.
By Claim~\ref{claimC-app} we have \eqref{AB<EKR}.

\subsection{The case
$\binom{n-1}{n-k}+\binom{n-i}{n-k-1}
<|\cA|\leq\binom{n-1}{n-k}+\binom{n-i+1}{n-k-1}$ for some $i$ with 
$4\leq i\leq i_0$}

Let $|\cA|=X+\binom x{n-k-2}$
where $X=\binom{n-1}{n-k}+\binom{n-i}{n-k-1}$ and $n-k-2\leq x\leq n-i$. 
Then, by Lemma~\ref{KK lemma}, 
$|\cB|\leq Y-\binom x{l-2}$, where $Y=\binom{n-1}{l-1}-\binom{n-i}{l-1}$. 
By Claim~\ref{claimA-app} (iii) we have \eqref{AB<EKR}.

\subsection{The case $\binom{n-1}{n-k}+\binom{n-3}{n-k-1}<|\cA|\leq
\binom{n-1}{n-k}+\binom{n-2}{n-k-1}$.
}

Let $|\cA|=X+\binom{x}{n-k-2}$, where $X=\binom{n-1}{n-k}+\binom{n-3}{n-k-1}$
and $x\geq n-k-2$.
Then, by Lemma~\ref{KK lemma}, $|\cB|\leq Y-\binom{x}{l-2}$,
where $Y=\binom{n-1}{l-1}-\binom{n-3}{l-1}$.
By Claim~\ref{claimA-app} (i) and (ii), we have
\eqref{AB<EKR} for $x\leq n-4$, 
and for $x\leq n-3$ and $\alpha<0.27$.
Thus, for the remaining, we may assume that
$\binom{n-1}{n-k}+\binom{n-3}{n-k-1}+\binom{n-4}{n-k-2}
<|\cA|\leq\binom{n-1}{n-k}+\binom{n-2}{n-k-1}$ and $\alpha\geq 0.27$.
Let $|\cA|=X+\binom x{n-k-3}$, where 
$X=\binom{n-1}{n-k}+\binom{n-3}{n-k-1}+\binom{n-4}{n-k-2}$
and $x\geq n-k-3$.
Then, by Lemma~\ref{KK lemma},
$|\cB|\leq Y-\binom x{l-3}$, where 
$Y=\binom{n-1}{l-1}-\binom{n-i}{l-1}-\binom{n-4}{l-2}$. 
In this case, by Claim~\ref{claimB-app} (ii), we have \eqref{AB<EKR}
for $x\leq n-3$ and $\alpha>0.23$.

\subsection{The case $\binom{n-1}{n-k}+\binom{n-2}{n-k-1}<|\cA|$.}\label{subsec3.4}

First suppose that $\binom{n-1}{n-k}+\binom{n-2}{n-k-1}<|\cA|\leq
\binom{n-1}{n-k}+\binom{n-2}{n-k-1}+\binom{n-3}{n-k-2}$.
Let $|\cA|=\binom{n-1}{n-k}+\binom{n-2}{n-k-1}+\binom{x}{n-k-2}$
for $n-k-2\leq x\leq n-3$.
Then, by Lemma~\ref{KK lemma}, $|\cB|\leq\binom{n-2}{l-2}-\binom{x}{l-2}$, and 
\eqref{AB<EKR} follows from Claim~\ref{claimA-app} (i).

Next suppose that 
$\binom{n-1}{n-k}+\binom{n-2}{n-k-1}+\binom{n-3}{n-k-2}<|\cA|
\leq\binom{n-1}{n-k}+\binom{n-2}{n-k-1}+\binom{n-3}{n-k-2}+\binom{n-4}{n-k-3}$.
Let $|\cA|=\binom{n-1}{n-k}+\binom{n-2}{n-k-1}+\binom{n-3}{n-k-2}
+\binom{x}{n-k-3}$ for $n-k-3\leq x\leq n-4$.
Then, by Lemma~\ref{KK lemma}, $|\cB|\leq\binom{n-3}{l-3}-\binom{x}{l-3}$, and
\eqref{AB<EKR} follows from Claim~\ref{claimB-app} (i).

Finally suppose that 
$\binom{n-1}{n-k}+\binom{n-2}{n-k-1}+\cdots+\binom{n-t}{n-k-t+1}<|\cA|
\leq\binom{n-1}{n-k}+\binom{n-2}{n-k-1}+\cdots+\binom{n-t-1}{n-k-t}$ 
for some $t\geq 4$. Then, by Lemma~\ref{KK lemma}, $|\cB|\leq\binom{n-t}{l-t}$.
In this case we generously estimate 
\begin{align*}
\lim_{n\to\infty}\frac{|\cA|}{\binom nk}\leq\alpha(1+\aa+\aa^2+\cdots+\aa^{t})
,\text{ and }
\lim_{n\to\infty}\frac{|\cB|}{\binom nl}\leq\beta^{t}.
\end{align*}
Let $\gamma:=1+\aa+\aa^2+\cdots+\aa^{t}$.
In view of Lemma~\ref{lemma:limit} it suffices to show that $\gamma\beta^{t-1}<1$. 
Recall from \eqref{beta<tilde beta} that $\beta\leq\tilde\beta=2-\sqrt{2}$. 
If $t=4$, then $\gamma=\frac{1-\aa^4}\alpha\leq 4$ and
$\beta^{t-1}\leq\tilde\beta^3<0.21$, so $\gamma\beta^{t-1}<1$.
Let $t\geq 5$.
Since $\gamma\beta^{t-1}\leq(t+1)\tilde\beta^{t-1}$ and
the RHS is decreasing in $t$, it is maximized at $t=5$. Thus
$\gamma\beta^{t-1}\leq 6\tilde\beta^4<1$.

This completes the proof of Lemma~\ref{main lemma}.

\section{Proof of Theorem~\ref{thm-unif}: Uniqueness of the optimal families}\label{sec4}
In the previous section we have proved Lemma~\ref{main lemma}, that is, 
if $|\cA|>\binom{n-1}{k-1}$ or
$|\cB|>\binom{n-1}{l-1}$, then $|\cA||\cB|<\binom{n-1}{k-1}\binom{n-1}{l-1}$
(under the assumptions of Theorem~\ref{thm-unif}).
Thus we have 
\begin{align}\label{ineq:uniqueness}
|\cA||\cB|\leq\binom{n-1}{k-1}\binom{n-1}{l-1},
\end{align}
and equality holds if and only if $|\cA|=\binom{n-1}{k-1}$ and
$|\cB|=\binom{n-1}{l-1}$. In this case we show that $\cA$ and $\cB$ are stars.

Suppose that equality holds in \eqref{ineq:uniqueness}. Then 
$|\cA^c|=|\cA|=\binom{n-1}{n-k}$ and, by Lemma~\ref{KK lemma},
\[
\binom{n-1}{l-1}=|\cB|\leq\binom nl-|\sigma_l(\cA^c)|\leq\binom nl-\binom{n-1}l
= \binom{n-1}{l-1}.
\]
Thus $|\sigma_l(\cA^c)|=\binom{n-1}l$ must hold, and it follows from the
Kruskal--Katona Theorem that this is possible only
when $\cA^c=\binom{[n]\setminus\{i\}}{n-k}$ for some $i\in[n]$.
In this case we have $\cA=\{A\in\binom{[n]}k:i\in A\}$ and
$\cB=\{B\in\binom{[n]}l:i\in B\}$, that is, both families are stars.

\section{Proof of Theorem~\ref{thm-p}: The measure version}\label{sec5}
The proof is essentially written in \cite{T0}, but for convenience
we include it here.

Recall that $\Delta$ is the set of $(\alpha,\beta)\in\Omega$ satisfying 
$\mu_{\alpha}(\cA_j)\mu_\beta(\cB_j)<\alpha\beta$ for all $j\geq 0$.
Let $(\alpha,\beta)\in\Delta$ be fixed.
Then we can choose $\epsilon>0$ sufficiently small so that
$\alpha-\epsilon<\alpha'<\alpha+\epsilon$ and 
$\beta-\epsilon<\beta'<\beta+\epsilon$ imply $(\alpha',\beta')\in \Delta$.
Let 
\begin{center}
$I:=((\alpha-\ee)n,(\alpha+\ee)n)\cap\N$
and $J:=((\beta-\ee)n,(\beta+\ee)n)\cap\N$.
\end{center}
Let $\cF$ and $\cG$ be cross-intersecting families in $2^{[n]}$.
Let $\cF^{(k)}:=\cF\cap\binom{[n]}k$ and $\cG^{(l)}:=\cG\cap\binom{[n]}l$.
If $k\in I$ and $l\in J$ then we can apply Theorem~\ref{thm-unif} to
$\cF^{(k)}$ and $\cG^{(l)}$, and we 
get $|\cF^{(k)}||\cG^{(l)}|\leq\binom{n-1}{k-1}\binom{n-1}{l-1}$
provided $n>n_0(\alpha,\beta)$.
We also use the fact that the binomial distribution $B(n,\alpha)$ is 
concentrated around $\alpha n$, which yields
$$\lim_{n\to\infty}\sum_{k\in I}\ts{\binom{n-1}{k-1}}\alpha^k(1-\alpha)^{n-k}
=\alpha,
\text{ and }\displaystyle
\lim_{n\to\infty}\sum_{k\not\in I}\ts{\binom{n}{k}}\alpha^k(1-\alpha)^{n-k}=0,$$
and the corresponding formulae for $\beta$ and $J$.
Thus, as $n\to\infty$, we have
\begin{align*}
\mu_{\alpha}(\cF) \mu_{\beta}(\cG) &\leq 
\bigg(\sum_{k\in I}|\,\cF^{(k)}|\,\alpha^k(1-\alpha)^{n-k}
+\sum_{k\not\in I}{\ts {\binom nk}}\alpha^k(1-\alpha)^{n-k}\bigg)\nonumber\\
&\qquad \times
\bigg(\sum_{l\in J}|\,\cG^{(l)}|\,\beta^l(1-\beta)^{n-l}
+\sum_{l\not\in J}{\ts {\binom nl}}\beta^k(1-\beta)^{n-l}\bigg)
\nonumber\\
&=
\bigg(\sum_{k\in I}|\,\cF^{(k)}|\,\alpha^{k}(1-\alpha)^{n-k}\bigg)
\bigg(\sum_{l\in J}|\,\cG^{(l)}|\,\beta^{l}(1-\beta)^{n-l}\bigg)
+o(1)\nonumber\\
&=
\sum_{k\in I}\sum_{l\in J}|\,\cF^{(k)}||\,\cG^{(l)}|\,
\alpha^{k}(1-\alpha)^{n-k}\beta^{l}(1-\beta)^{n-l}+o(1)\nonumber\\
&\leq
\sum_{k\in I}\sum_{l\in J}
{\ts {\binom{n-1}{k-1}}}
{\ts {\binom{n-1}{l-1}}}
\alpha^{k}(1-\alpha)^{n-k}
\beta^{l}(1-\beta)^{n-l}+o(1)\nonumber\\
&=
\bigg(\sum_{k\in I}
{\ts {\binom{n-1}{k-1}}}
\alpha^{k}(1-\alpha)^{n-k}\bigg)
\bigg(\sum_{l\in J}
{\ts {\binom{n-1}{l-1}}}
\beta^{l}(1-\beta)^{n-l}\bigg)+o(1)\nonumber\\
&=\alpha\beta+o(1),
\end{align*}
and $m(\alpha,\beta)\leq\alpha\beta$.
Then, by the fact that $m(n,\alpha,\beta)$ is monotone increasing in $n$ 
(see the footnote at the definition of $m(\alpha,\beta)$), it follows that
$m(n,\alpha,\beta)=\alpha\beta$ for all $n\geq 1$. 
\qed

\section{Proof of Theorem~\ref{thm1}: A sketch}\label{sec6}
The proof is similar to the proof of the main result in \cite{MT}, 
and much easier than that of Theorem~\ref{thm-unif}. So we only include a 
sketch here.

Recall that the case $|\cB|>\binom{n-1}{l-1}$ is already settled by 
Claim~\ref{easy case},
and we may assume that $|\cA|>\binom{n-1}{k-1}$.
Let
\[
|\cA|=\sum_{i=1}^t\binom{n-i}{n-k-i+1}+\binom x{n-k-t}
\]
for some $t$ with $1\leq t\leq n-k$ and $x$ with $n-k-t\leq x\leq n-t-1$.
Then, by Lemma~\ref{KK lemma}, we have
\[
|\cB|\leq\binom{n-t}{l-t}-\binom x{l-t}. 
\]
We need to show the following.
\begin{lemma}\label{lemma for thm1}
For $1\leq t\leq n-k$ and $n-k-t\leq x\leq n-t-1$,
\[
\left(\sum_{i=1}^t\binom{n-i}{n-k-i+1}+\binom x{n-k-t} \right)
\left(\binom{n-t}{l-t}-\binom x{l-t}\right)<\binom{n-1}{k-1}\binom{n-1}{l-1}.
\]
\end{lemma}
The proof of the lemma breaks down into four cases: $t=1$, $t=2$, $t=3$,
and $t\geq 4$. The case $t\geq 4$ is easy, and one can argue as in
Subsection~\ref{subsec3.4}. The proof of the other three cases is similar,
and we only present the case $t=1$ here. In this case Lemma~\ref{lemma for thm1}
reads as follows.

\begin{claim}
 For $n\geq 2$ and $n-k-1\leq x\leq n-2$ we have
\[
 \left(\binom{n-1}{k-1}+\binom x{n-k-1}\right)
 \left(\binom{n-1}{l-1}-\binom x{l-1}\right)<\binom{n-1}{k-1}\binom{n-1}{l-1}.
\]
\end{claim}

\begin{proof}
We argue exactly in the same way as in the proof of Claim~\ref{claimC}. Then 
we get \eqref{cond:claimC} for $i=2$, which reduces to
\[
 (n-k)\sum_{j=0}^{l-2}\frac1{n-2-j}\geq(n-l)
\sum_{j=0}^{n-k-2}\frac1{n-2-j}.
\]
This is the opposite inequality to \eqref{C2}, and this
contradiction proves the claim.
\end{proof} 

\section*{Acknowledgment}
I would like to thank the two referees for their helpful comments.

\end{document}